\documentclass{amsart}

\usepackage[paperwidth=210mm,paperheight=297mm,hmargin={35mm,35mm},vmargin={35mm,35mm}]{geometry} 
\usepackage{amsthm,amssymb,latexsym,amsmath}
\usepackage{mathrsfs}
\usepackage{graphicx}
\usepackage{hyperref}

\input xy
\xyoption{all}
\usepackage[all]{xy}
\usepackage{amsfonts,color, soul}

\usepackage{mathrsfs}
\usepackage{comment}

\renewcommand{\dim}{\mathrm{dim}}

\newcommand{\Sing}{{\rm{Sing}}}
\newcommand{\codim}{{\rm{codim}}}
\newcommand{\kerr}{{\rm{Ker}}}

\newtheorem{nada}{Nada}[section]
\newtheorem{definition}[nada]{Definition}

\newtheorem{corollary}[nada]{Corollary}
\newtheorem{theorem}[nada]{Theorem}
\newtheorem*{thm*}{Theorem}
\newtheorem{lemma}[nada]{Lemma}

\newtheorem{rmk}[nada]{Remark}
\newtheorem{example}[nada]{Example}

\newcommand{\bc}{\begin{center}}
\newcommand{\ec}{\end{center}}
\newcommand{\noi}{\noindent}

\newcommand{\rank}{{\rm{rank}}}
\theoremstyle{plain}

\begin{document}

\title{Inequalities and enumerative formulas for
flags of Pfaff systems}
\hyphenation{ho-mo-lo-gi-cal}
\hyphenation{fo-lia-tion}

\begin{abstract}

In this work, we study inequalities and enumerative formulas for flags of Pfaff Systems on $\mathbb{P}^n_{\mathbb{C}}$.
More specifically, we establish a bound for the number of independent twisted $r$-forms that leave invariant a one-dimensional holomorphic foliation and deduce inequalities that relate the degrees in the flags, which can be interpreted as a version of the Poincar\'e problem for flags. 
Moreover, by restricting to a flag of specific holomorphic foliations/distributions, we obtain inequalities involving the degrees. 
As a consequence, we obtain stability results for the tangent sheaf of some rank two holomorphic foliations/distributions.
\end{abstract}

\author{Miguel Rodr\'iguez Pe\~na}
\address{Arnulfo Miguel Rodr\'iguez Pe\~na \\ ICEX-UFMG, Departamento de Matem\'atica,
Belo Horizonte MG, Brazil, CEP 31270-901. \href{https://orcid.org/0009-0005-0837-4466}{ORCID ID: 0009-0005-0837-4466}.}
\email{amrp2024@ufmg.br}

\author{Fernando Louren\c co}
\address{Fernando Louren\c co \\ DMM-UFLA, Campus Universit\'ario, 
Lavras MG, Brazil, CEP 37200-000. \href{https://orcid.org/0009-0003-1837-5417}{ORCID ID: 0009-0003-1837-5417}.}
\email{fernando.lourenco@ufla.br}
\thanks{ }

\maketitle

\section{Introduction}
  
The concept of a flag of holomorphic foliations is a relatively recent topic, and many authors have studied it; see, for instance,  
\cite{BCL,SoaCo,Fei}, and \cite{Mol}. By a {\it flag} of holomorphic distributions on a complex manifold $M$, 
$$\mathcal{D}_{k_1}\prec\mathcal{D}_{k_{2}}\prec\cdots\prec\mathcal{D}_{k_m},$$ 
we mean a collection of distributions of dimensions
$$1 \leq k_1 < k_2 < \cdots < k_m< \mathrm{dim}(M),$$ 
such that, at each point $p\in M$ where the distributions are regular, 
$$\mathcal{D}_{k_i,\,p}\subset\mathcal{D}_{k_j,\,p}\,\,\,\mbox{whenever}\,\,\,i<j.$$ 
In particular, under the assumption of integrability,
this implies that the leaves of $\mathcal{D}_{k_i}$ are contained in the leaves of $\mathcal{D}_{k_j}$. 
All these concepts are explained in Section \ref{Pre}.

In this paper, we extend the concept of a flag of holomorphic foliations/distributions to that of a flag of Pfaff systems. 
Using this new approach, we study inequalities and enumerative problems associated with such flags. The inequalities presented 
in this work are motivated by the so-called Poincar\'e problem for foliations. 

Before stating the results, we recall that a {\it Pfaff system} $\mathscr{F}$ on a complex projective space $\mathbb{P}^n$ is defined by a nontrivial global section 
$$\omega_{\mathscr{F}} \in H^{0}(\mathbb{P}^{n}, \Omega_{\mathbb{P}^{n}}^{k} \otimes \mathcal{O}_{\mathbb{P}^{n}}(l)),$$ 
where $\Omega_{\mathbb{P}^{n}}^{k}$ denotes the sheaf of holomorphic $k$-forms on $\mathbb{P}^n$. The number $k$, with $1\leq k\leq n-1$,
is called the codimension of $\mathscr{F}$. A {\it holomorphic distribution} is a twisted $r$-form that is locally decomposable outside its singular set, and a {\it holomorphic foliation} is a distribution that is integrable in the sense of Frobenius.


Motivated by the question of algebraic integrability of foliations, Henri Poincar\'e studied in \cite{Poin} the problem of determining whether a holomorphic foliation $\mathcal{F}$ on the complex projective plane admits a rational first integral. Poincar\'e observed that, 
to address this problem, it suffices to find a bound for the degree of the generic $\mathcal{F}$-invariants curves.
The question of bounding the degree of an algebraic variety invariant under a foliation on projective space, in terms of the degree of the foliation, is known as the {\it Poincar\'e problem}.

On the other hand, the Poincar\'e problem can be interpreted in terms of flags. Specifically, if we consider a flag of holomorphic foliations on projective space $\mathbb{P}^{n}$, to what extent can the degrees of the foliations comprising the flag be related?
Several topics closely related to flags arise naturally in the theory of holomorphic foliations.



Next, we present some of the main results of this work, see also Section \ref{Ineq}.
The following theorem provides a count of the number of independent twisted $r$-forms that leave invariant a one-dimensional holomorphic foliation.

%
%
%
%
%
%
%

\begin{theorem}\label{thm 2} Let $\mathcal{F}$ be a one-dimensional holomorphic foliation of degree $d$ on $\mathbb{P}^{n}$ whose singular set consists only of isolated singularities. Then, for each integer number $m$, there exist
$\wp(d,m,r)$ independent twisted $r$-forms of degree $m$ that leave $\mathcal{F}$ invariant, where
$$\footnotesize\wp(d,m,r)=\left\{\begin{array}{ccccc}

0 & if & m+1\leq d \\ \\

\displaystyle\sum_{i=1}^{j}(-1)^{i+1}\binom{m+n+1-id}{m+r+i+1-id} \binom{m+r+i-id}{r+i}  & if &  \begin{array}{cc}
                              \mbox{} \\
                              jd< m+1 \leq (j+1)d,  \\ 
                              j=1,\ldots,n-r-2
                              \end{array} \\ \\

\displaystyle\sum_{i=1}^{n-r}(-1)^{i+1}\binom{m+n+1-id}{m+r+i+1-id} \binom{m+r+i-id}{r+i} & if &  (n-r-1)d<m+1.

\end{array}\right.$$

\end{theorem}

In the following corollary, we investigate the relationship between the degrees of the Pfaff systems involved in the flag

\begin{corollary}\label{thm 1} Let $\mathcal{F}\prec\mathcal{G}$ be a flag on $\mathbb{P}^{n}$, where $\mathcal{F}$ is a one-dimensional holomorphic foliation which its singular set has only isolated singularities, and $\mathcal{G}$ is a Pfaff system. Then
$$\deg (\mathcal{F}) \leq \deg (\mathcal{G}).$$
\end{corollary}

This result can also be derived from Bott's formula. Indeed, any twisted $r$-form $\omega \in H^{0}(\mathbb{P}^{n}, \Omega_{\mathbb{P}^{n}}^{r}(r+1+\deg(\omega)))$ such that $i_{X}\omega=0$ for some vector field $X \in H^{0}(\mathbb{P}^{n}, T_{\mathbb{P}^{n}}(d-1))$ of degree $d$, can be write as $\omega=i_{X}\theta$, where $\theta \in H^{0}(\mathbb{P}^{n}, \Omega_{\mathbb{P}^{n}}^{r+1}(r+2+\deg(\theta)))$. 
In this case, the degrees satisfy the relation $\deg(\omega)=\deg(\theta)+d$.


Note that Theorem \ref{thm 2} gives a finitely many flags which is reaches the maximal bound in Corollary \ref{thm 1}, proving that this bound is sharp (see also Example \ref{Ex2}).

\begin{corollary}\label{coro001} Let $\mathcal{D}$ be a holomorphic distribution such that 
$T_\mathcal{D}\cong\bigoplus T_{\mathcal{F}_i}$, where each $\mathcal{F}_i$ 
is a one-dimensional holomorphic foliation which its singular set has only isolated singularities. 
Let $\mathcal{G}$ be a Pfaff system such that leaves $\mathcal{D}$ invariant. Then 
$$\deg(\mathcal{D})\leq\mathrm{dim}(\mathcal{D})\cdot\deg(\mathcal{G}).$$	
\end{corollary}

In \cite[Theorem 2]{CuPe} (see also Theorem $1$), the authors show that the split condition of $T_\mathcal{D}$ 
in Corollary \ref{coro001} is an open condition for $n\geq 4$, $\mathrm{dim}(\mathcal{D})\leq n-2$, and 
$\mathrm{codim}\hspace{-0.2mm}\big(\mathrm{Sing}(\mathcal{D})\big)\geq 3$.

\begin{rmk} We observe that a holomorphic distribution $\mathcal{D}$ on $\mathbb{P}^{n}$ is nonsingular if and only if $n$ is odd, 
$\codim(\mathcal{D})=1$ and $\deg(\mathcal{D})=0$, see \cite[p.36]{Coo} and \cite{GHS}. In this work, we will assume that all distributions
have a nonempty singular set.
\end{rmk}

In the next result, we consider a dual version of the above result in Theorem \ref{thm 2}. Namely, assuming a fixed codimension one distribution, we count-under generic conditions-the number of tangent, linearly independent vector fields.

\begin{theorem}\label{thm 3} Let $\mathcal{D}$ be a holomorphic distribution of codimension one and degree $m$ on $\mathbb{P}^{n}$ which its singular set has only isolated singularities. Then for each integer number $d$ subject to conditions

\begin{enumerate} 
\item[(1)] $m+1\leq d < 2\left(m+1\right)$,
\item[(2)] $d \neq \big(\frac{n}{2}\big)m$ if $n$ is even,
\item[(3)] $d \neq \left(\frac{n-1}{2}\right)m-1$ and 
$d \neq \left(\frac{n+1}{2}\right)m+1$ if $n$ is odd,
\end{enumerate}
\noindent there are
$$\binom{d-m+n}{d-m+n-2}\binom{d-m+n-3}{n-2}$$
\vskip 0.1cm
\noindent independent vector fields $X$ of degree $d$ on $\mathbb{P}^{n}$ invariants by $\mathcal{D}$. Moreover, if $d < m+1$, we do not have nontrivial vector field $X$ of degree $d$ invariant by $\mathcal{D}$ in the case $n$ even, or 
such that $d \neq \left(\frac{n-1}{2}\right)m-1$ in the case $n$ odd.
\end{theorem}

Note that for $n=3$, it suffices to consider the hypothesis $m+1\leq d < 2m+1$, and thus there are 
$$\frac{(d-m+3)(d-m+2)(d-m)}{2}$$
\noindent independent vector fields $X$ of degree $d$ on $\mathbb{P}^3$ invariants by $\mathcal{D}$. On the other hand, 
if $d< m+1$ and $d\neq m-1$, we do not have vector field $X$ of degree $d$ invariant by $\mathcal{D}$.


\begin{rmk}
In \cite[Theorem 1.1]{SoaCo} the authors consider a flag $\mathcal{F} \prec \mathcal{G}$ on $\mathbb{P}^{n},$ with 
$\dim (\mathcal{F}) = \codim(\mathcal{G})=1$, $\mathrm{Sing}(\mathcal{G})$ isolated, 
and under certain hypotheses, they show that $\deg (\mathcal{G}) \leq \deg (\mathcal{F}) - 1$.
We see in particular that Theorem \ref{thm 3} recovers this result with fewer hypotheses.
\end{rmk}

Combining Corollary \ref{thm 1} and Theorem \ref{thm 3}, we get

\begin{corollary}\label{coro1.5} We do not have flag $\mathcal{F}\prec \mathcal{G}$ on $\mathbb{P}^{n}$ satisfying the conditions
\begin{enumerate}
 \item[a)] $\dim(\mathcal{F}) = \codim(\mathcal{G}) = 1;$
 \item[b)] both $\mathcal{F}$ and $\mathcal{G}$ have only isolated singularities;
 \item[c)] $\deg(\mathcal{F}) \neq \left(\frac{n-1}{2}\right) \deg(\mathcal{G}) -1,$ if $n$ is odd.
\end{enumerate}
\end{corollary}

In the next corollary we partially recover the result in \cite[p.9014]{OmCoJa1} using Theorem \ref{thm 3}.

\begin{corollary} Let $\mathcal{D}$ be a codimension one holomorphic distribution on $\mathbb{P}^{3}$. If 
$\mathrm{Sing}(\mathcal{D})$ has only isolated singularities, then $\mathcal{D}$ is semistable, and $\mathcal{D}$ 
is stable for $\deg (\mathcal{D})\neq 2$. 
\end{corollary}

\begin{proof} Let $\omega \in H^{0}(\mathbb{P}^{3}, \Omega_{\mathbb{P}^{3}}^{1}(m+2))$ be a global section inducing $\mathcal{D}$, 
$m=\deg(\mathcal{D})$. Then $\mu(T_{\mathcal{D}})=\frac{2-m}{2}$. Given 
$\mathcal{F}$ any nontrivial subsheaf of $\mathcal{D}$, it can be represented by a global section 
$X \in H^{0}(\mathbb{P}^{3}, T_{\mathbb{P}^{3}}(d-1))$, $d\geq0$, and so its slope is $\mu(T_{\mathcal{F}})=1-d$.
We have two possibilities, either $m+1\leq d$ or $d<m+1$. The inequality $m+1 \leq d$ implies $\mu(T_{\mathcal{F}}) < \mu(T_{\mathcal{D}})$.

On the other hand, in the inequality $d<m+1$, if $d \neq m-1$, then by Theorem \ref{thm 3} we do not have nontrivial subsheaf of $\mathcal{D}$ of rank one, and so $d=m-1$, it follows $\mu(T_{\mathcal{F}}) < \mu(T_{\mathcal{D}})$ for $m>2$, and 
$\mu(T_{\mathcal{F}})=\mu(T_{\mathcal{D}})=0$ for $m=2$. Note that $m \neq 1$ by Corollary \ref{coro1.5}.
\end{proof}

At the end of the paper, we consider the Poincar\'e problem for flags in certain classes of foliations,
such as logarithmic and pull-back foliations, as well as decomposable/split distributions. 
For instance, in the logarithmic case, we obtain the following results.

\begin{theorem}\label{A} Let $\mathcal{G}\in\mathcal{L}(n,p,d_1,\ldots,d_r)$, $r\geq p+1$, be a logarithmic foliation on 
$\mathbb{P}^{n}$, $n\geq p+2$, induced in homogeneous coordinates by a $p$-form
$$\omega=F_1\cdots F_r\cdot\hspace{-1mm}\sum_{\textsc{I}=(i_1<\cdots<i_p)}
\lambda_{\,\textsc{I}}\frac{dF_{i_1}}{F_{i_1}}\wedge\cdots\wedge\frac{dF_{i_p}}{F_{i_p}},$$
for some irreducible homogeneous polynomials $F_i$ of degree $d_i\geq 1$ and $\lambda_{\textsc{\,I}}\neq 0$.  
Set $\left|d\right|=\sum d_i$, $V=V_{i_1,\ldots,i_{p+1}}=V(F_{i_1},\ldots,F_{i_{p+1}})$ and $R=\mathrm{reg}\big(\mathrm{Sing}(V)\big)$.
Let $\mathcal{F}\prec\mathcal{G}$ be a flag of holomorphic foliations on $\mathbb{P}^{n}$. 
\vskip 0.2cm
\begin{enumerate}
	\item If some $\left(F_{i}=0\right)$ is smooth and $\dim(\mathcal{F})=1$, then 
$$\deg(\mathcal{G})\leq\deg(\mathcal{F})+\left|d\right|-d_i-p.$$	
	\item If some $\left(F_{i}=0\right)$ is a normal crossing hypersurface and $\dim(\mathcal{F})=1$, then
$$\deg(\mathcal{G})\leq\deg(\mathcal{F})+\left|d\right|-d_i+n-p-1.$$		
\item Suppose that $V$ is a complete intersection curve $(n=p+2)$ and $\dim(\mathcal{F})=1$.
\vskip 0.2cm
	\begin{enumerate}
		\item[a)] If $V$ is a smooth and $V\not\subset\mathrm{Sing}(\mathcal{F})$, then 
$$\deg(\mathcal{G})\leq\deg(\mathcal{F})+\left|d\right|-\sum_{k=1}^{p+1} d_{i_k}.$$ 
		\item[b)] If $V$ is reduced with at most ordinary nodes as singularities, and $V\not\subset\mathrm{Sing}(\mathcal{F})$, then
$$\deg(\mathcal{G})\leq\deg(\mathcal{F})+\left|d\right|-\sum_{k=1}^{p+1} d_{i_k}+1.$$
	\end{enumerate} 
	\item Suppose that $V$ is a complete intersection and $\mathrm{codim}(\mathcal{F})=p+1$. 
\vskip 0.2cm
\begin{enumerate}
	\item[a)] Assume $V$ is reduced and 
$\mathrm{dim}\hspace{-0.2mm}\big(\mathrm{Sing}(\mathcal{F})\cap V\big)< n-p-1$. If
$R\leq\sum_{k=1}^{p+1} d_{i_k}-p-2$, then 
$$\deg(\mathcal{G})\leq\deg(\mathcal{F})+\left|d\right|-\sum_{k=1}^{p+1} d_{i_k};$$ 
and if $R>\sum_{k=1}^{p+1} d_{i_k}-p-2$, then
$$\deg(\mathcal{G})\leq \frac{1}{2}\hspace{-0.2mm}\big(\deg(\mathcal{F})+R+1\big)+\left|d\right|-\sum_{k=1}^{p+1} d_{i_k}.$$	
  \item[b)] If $V$ is nonsingular in codimension $1$ and $V\not\subset\mathrm{Sing}(\mathcal{F})$, then 
$$\deg(\mathcal{G})\leq\deg(\mathcal{F})+\left|d\right|-\sum_{k=1}^{p+1} d_{i_k}+1.$$
\end{enumerate}
\end{enumerate}
\end{theorem}
\vskip 0.2cm

As a consequence, we obtain results concerning the Mumford-Takemoto stability of the rank $2$ tangent sheaves $T_{\mathcal{G}}$; see Section \ref{Ineq}. 

\bigskip

The paper is organized as follows. First, to make this paper as self-contained as possible, we provide all the necessary definitions and considerations in Section \ref{Pre}. The proofs of our main results appear in Sections \ref{Proof}, \ref{Proof2}, and \ref{Ineq}. 
In Section \ref{Exa}, we present some examples.

\bigskip
\subsection*{Acknowledgments}
The authors wish to express their gratitude to Maur\'icio Corr\^ea and
Alan Muniz for helpful comments and suggestions. The authors also thank the anonymous referees for providing many suggestions that helped improve the paper's presentation and for encouraging them to extend Theorem 1.1 to more cases. The authors were partially supported by the FAPEMIG [grant number 38155289/2021]. 
\bigskip

\section{Preliminaries} \label{Pre}

In this paper, we introduce a new approach to the flag considering instead only holomorphic foliations, Pfaff systems (possibly non locally decomposable). For this purpose, let us consider some notations.
\subsection{Pfaff systems on $\mathbb{P}^{n}$}
In this Subsection, we deal with Pfaff systems on the complex projective space $\mathbb{P}^{n}$. In its original meaning,
a Pfaffian system is a system of Pfaffian equations $\omega_{1} = \cdots = \omega_{k}= 0$, where each $\omega_{i}$ is a holomorphic $1$-form. Denote by $\Omega_{\mathbb{P}^{n}}^{k}$ the sheaf of germs of 
$k$-forms on $\mathbb{P}^{n}$. Given $\mathcal{O}_{\mathbb{P}^{n}}(l)$ 
a line bundle on $\mathbb{P}^{n}$ we use the notation (twisted $k$-forms)
$$\Omega_{\mathbb{P}^{n}}^{k}(l) := \Omega_{\mathbb{P}^{n}}^{k} \otimes \mathcal{O}_{\mathbb{P}^{n}}(l).$$

\begin{definition}\label{defiPfaff} A Pfaff system $\mathscr{F}$ of codimension $k$ on $\mathbb{P}^{n}$ is induced by a 
nontrivial section (a twisted $k$-form) 
$$\omega_{\mathscr{F}} \in H^{0}(\mathbb{P}^{n}, \Omega_{\mathbb{P}^{n}}^{k}(l)).$$
\end{definition}

The singular set of $\mathscr{F}$ is naturally defined as
$$\Sing(\mathscr{F}) = \{\,p \in \mathbb{P}^{n};\,\omega_{\mathscr{F}}(p) = 0 \}.$$

We remark that a {\it holomorphic distribution} of codimension $k$ is a twisted $k$-form that is locally decomposable outside its singular set, i.e., for any point $p \in \mathbb{P}^{n}\setminus \Sing(\mathscr{F})$, there exists an open neighborhood $U$ of $p$ and holomorphic 
$1$-forms $\omega_{1}, \ldots , \omega_{k}$ on $U$ such that
\begin{equation}\label{eq 001}
\omega_{\mathscr{F}}|_{U}= \omega_{1} \wedge \cdots \wedge \omega_{k}.
\end{equation}

We note that, for $k=2$, the decomposable condition can be reformulated as follows 
(see \cite[Proposition 1, p.5]{CerLN}), \textit{ $\omega|_{U}$ is decomposable if and only if $\omega|_{U} \wedge \omega|_{U} =0$}. 

Again we observe that a {\it holomorphic foliation} of codimension $k$ is an integrable distribution, where the integrability condition is given by
\begin{equation}\label{eq002}
d\omega_{i} \wedge \omega_{1} \wedge \cdots \wedge \omega_{k} = 0,
\end{equation}

\noindent for $i=1, \ldots , k$. 
For a background in holomorphic distribution and foliation, we refer \cite{BB1,Baum} and \cite{OmCoJa1} (and references therein).
Throughout the text, we consider reduced holomorphic foliations, i.e., 
holomorphic foliations whose singular set has codimension greater than $1$.

Given a twisted $k$-form on $\mathbb{P}^{n}$, it can be determined by the following data, see \cite{C}.

\begin{enumerate}
\item[a)] an open covering $\{U_{\alpha} \}_{\alpha \in \Lambda}$ of $\mathbb{P}^{n};$
\item[b)] holomorphic $k$-forms $\omega_{\alpha} \in H^{0}(U_{\alpha}, \Omega_{U_{\alpha}}^{k})$, satisfying
$$\omega_{\alpha} = h_{\alpha \beta}\omega_{\beta} \ \ \mbox{on} \ \ U_{\alpha}\cap U_{\beta} \neq \emptyset,$$
\noindent where $h_{\alpha \beta} \in \mathcal{O}_{\mathbb{P}^{n}}(U_{\alpha}\cap U_{\beta})^{\ast}$ 
determines a cocycle representing $\mathscr{F}$. 
\end{enumerate}
\bigskip
Let $\mathscr{F}$ be a Pfaff system of codimension $k$ induced by the twisted $k$-form 
$\omega \in H^{0}(\mathbb{P}^{n}, \Omega_{\mathbb{P}^{n}}^{k}(l))$
and $i: H \simeq \mathbb{P}^{k}\hookrightarrow \mathbb{P}^{n}$ be a generic non-invariant linearly embedded subspace. 
Then we get the non-trivial section as follow
$$i^{\ast}\omega \in H^{0}(H, \Omega_{H}^{k}(l)) \simeq H^{0}(\mathbb{P}^{k}, \mathcal{O}_{\mathbb{P}^{k}}(-k-1+l)).$$


The \textit{degree of the Pfaff system} $\mathscr{F}$, denoted by $\deg(\mathscr{F})$, is defined as 


$$\deg(\mathscr{F}) = -k-1+l.$$

\noindent Then we can consider $\mathscr{F}$ on $\mathbb{P}^{n}$ of codimension $k$ given by
$$\omega \in H^{0}(\mathbb{P}^{n}, \Omega_{\mathbb{P}^{n}}^{k}(m+k+1)),$$

\noindent where $m = \deg (\mathscr{F})$. The Euler sequence implies that a Pfaff system of codimension $k$ and degree $m$ can be induced by a polynomial differential $k$-form $\omega$ with homogeneous coefficients of degree $m+1$ on $\mathbb{C}^{n+1}$ such that 
$i_{\vartheta}\omega=0$, where $\vartheta$ is the radial vector field defined by

$$\vartheta = z_{0}\dfrac{\partial}{\partial z_{0}} + \cdots + z_{n}\dfrac{\partial}{\partial z_{n}},$$

\noi and $z_{0}, \ldots,  z_{n}$ are homogeneous coordinates on $\mathbb{P}^{n}$. To general theory of Pfaff systems we refer \cite{C,CJ,CoMa,CoFeVer} and \cite{CEKLE}.

\subsection{Flag of Pfaff systems} Let us introduce a new notion of flag, the flag of Pfaff systems on $\mathbb{P}^{n}$, 
that can be considered on any complex manifold.

\begin{definition} Let $\mathcal{D}_{1}$ and $\mathcal{D}_{2}$ be two holomorphic distributions of dimensions $k_{1}$ and $k_{2}$ respectively, $k_1<k_2$, on $\mathbb{P}^{n}$. We say that $\mathcal{D}_{1} \prec \mathcal{D}_{2} $ is a $2$-flag of holomorphic distributions if $\mathcal{D}_{1}$ is a coherent sub $\mathcal{O}_{\mathbb{P}^{n}}$-module of $\mathcal{D}_{2}$. Furthermore if each $\mathcal{D}_{i}$ is integrable as in (\ref{eq002}) we say that we have a $2$-flag of holomorphic foliations. 
\end{definition}

We remark that when we have integrability, the flag condition implies that the leaves of foliation $\mathcal{D}_{1}$ are contained in leaves of $\mathcal{D}_{2}$ outside their singular sets. In this case, the singular set $\mathrm{Sing}(\mathcal{D}_2)$ is invariant by
$\mathcal{D}_{1}$; see \cite{Yo}. For more details about flags of holomorphic foliations/distributions, 
see for instance \cite{BCL,SoaCo,Fei} and \cite{Mol}.

\begin{definition} Let $\mathcal{F}$ be a foliation of codimension $n-1$ induced by a vector field 
$X \in H^{0}(\mathbb{P}^{n}, T_{\mathbb{P}^{n}}(d-1))$. Let $\mathcal{G}$ be a Pfaff system of codimension $k$ 
induced by $\omega\in H^{0}(\mathbb{P}^{n}, \Omega_{\mathbb{P}^{n}}^{k}(l))$, with $1\leq k <n-1$. 
We say that $\mathcal{F}\prec\mathcal{G}$ is a $2$-flag if 
$$i_{X}\omega = 0.$$
\end{definition}

In a more general case
\begin{definition} Let $\mathscr{F}_{1}$ and $\mathscr{F}_{2}$ be two Pfaff systems of codimensions $k_{1}$ 
and $k_{2}$ respectively, $k_{2} < k_{1}$, on $\mathbb{P}^{n}$, induced by 
$\omega_i\in H^{0}(\mathbb{P}^{n}, \Omega_{\mathbb{P}^{n}}^{k_i}(l_i))$, $i=1,2$. 
We say that $\mathscr{F}_{1}\prec\mathscr{F}_{2}$ is a $2$-flag if outside their singular sets, we have
$$\kerr(\omega_{1}) \subset \kerr(\omega_{2}),$$
\noindent where
$$\kerr(\omega)=\left\{v\in T_{\mathbb{P}^{n}}\,;\,i_v\omega=0\right\}.$$
\end{definition}

Given a flag $\mathscr{F}_{1}\prec\mathscr{F}_{2}$ of Pfaff systems, we say that $\mathscr{F}_{2}$ leaves $\mathscr{F}_{1}$ invariant or that $\mathscr{F}_{1}$ is tangent to $\mathscr{F}_{2}$, or that $\mathscr{F}_{1}$ is invariant by 
$\mathscr{F}_{2}$. Naturally, we define the singular set of the flag $\mathscr{F}_{1}\prec\mathscr{F}_{2}$ as the following analytic set
$$\Sing(\mathscr{F}_{1}) \cup \Sing(\mathscr{F}_{2}).$$

Consider the flag of holomorphic foliations $\mathcal{F}_{1}\prec\mathcal{F}_{2}$. Thus, each foliation induces a short exact sequence of sheaves for $i=1,2.$
$$0 \longrightarrow \mathcal{F}_{i} \longrightarrow T_{\mathbb{P}^{n}} \longrightarrow N_{i} \longrightarrow 0,$$

\noindent with $N_{i}= T_{\mathbb{P}^{n}}/\mathcal{F}_{i}  $ the normal sheaf of the foliation $\mathcal{F}_{i}.$ If we denote by $N_{12}$ the relative quotient sheaf $\mathcal{F}_{2}/ \mathcal{F}_{1}$ and by $N_{\mathcal{F}}$ the normal sheaf $N_{12}\oplus N_{1}$, we have the turtle diagram

$$ \xymatrix{ 0 \ar[rd]  &  & 0 \ar[ld] &  & 0  \\
   & \mathcal{F}_{1} \ar[rd] \ar[dd] &  & N_{2} \ar[lu] \ar[ru] &  \\
   &  & T_{\mathbb{P}^{n}} \ar[ru] \ar[rd]  &  &  \\
   & \mathcal{F}_{2} \ar[ru]   \ar[rd] &  & N_{1} \ar[uu]  \ar[rd] &  \\
 0 \ar[ru] &  & N_{12} \ar[ru] \ar[rd]  &  & 0 \\
  & 0 \ar[ru] &   &  0 &  }$$

This is very useful to compute characteristic classes of flag, see \cite{BCL,CoLo,Su1} and \cite{Suwa2}. 

\subsection{Mumford-Takemoto stability}\label{stab} Let $L$ be a very ample line bundle on a $n$-dimensional smooth projective variety $X$. 
The slope of a torsion free coherent sheaf $E$ on $X$ is defined by
$$\mu_L(E) = \dfrac{c_1(E) \cdot L^{n-1}}{\rank(E)}.$$
Recall that $E$ is said to be stable (respectively semistable) if every proper nontrivial subsheaf $F\subset E$ with $\rank(F)<\rank(E)$
satisfies $\mu_L(F)<\mu_L(E)$ (respectively $\mu_L(F)\leq\mu_L(E)$). 
We will say that a holomorphic foliation $\mathcal{F}$ on $X$ is stable (respectively semistable) when its tangent sheaf $T_{\mathcal{F}}$
is stable (respectively semistable). 

When $X=\mathbb{P}^{n}$, we take $L=\mathcal{O}_{\mathbb{P}^{n}}(1)$ and write $\mu=\mu_L$. Let $\mathcal{F}$ be a holomorphic foliation of codimension $k$ on $\mathbb{P}^{n}$ given by $\omega \in H^{0}(\mathbb{P}^{n}, \Omega_{\mathbb{P}^{n}}^{k}(\deg(\mathcal{F})+k+1))$. 
Since the singular set of $\mathcal{F}$ has codimension at least two we obtain the adjunction formula 
$$K_{\mathbb{P}^{n}}=K_{\mathcal{F}}\otimes \det N^{\ast}_{\mathcal{F}},$$
where $K_{\mathcal{F}}=\left(\wedge^{n-k}\,T^{\ast}_{\mathcal{F}}\right)^{\ast\ast}$ is the canonical bundle and $N_{\mathcal{F}}$ the normal bundle of $\mathcal{F}$. Note that 
$\det(N_{\mathcal{F}})=\mathcal{O}_{\mathbb{P}^{n}}(\deg(\mathcal{F})+k+1)$.
In particular, we have
$$\mu(T_{\mathcal{F}})=\frac{\dim(\mathcal{F})-\deg(\mathcal{F})}{\dim(\mathcal{F})}.$$
Therefore $\mathcal{F}$ is stable (respectively semistable) if and only if for every distribution $\mathcal{D}$ tangent to $\mathcal{F}$ 
we have $\frac{\deg(\mathcal{F})}{\dim(\mathcal{F})}<\frac{\deg(\mathcal{D})}{\dim(\mathcal{D})}$ (respectively $\leq$).
For more details about stability and holomorphic foliations, see for instance \cite{OmCoJa1, LoJoTo} and \cite{OSS}.

\subsection{Bott's formula}

To finish the preliminary part, we describe a little bit about the dimension of certain cohomology group, namely, the classical Bott's formulas, see \cite{SoaCo} and \cite{OSS}. Consider, for this purpose, $n$ a positive number, $p,q,r$ and $s$ nonnegative numbers and $k$ and 
$t$ integer numbers, then with the notation
$$h^{q}(\mathbb{P}^{n}, \Omega_{\mathbb{P}^{n}}^{p}(k))=\dim H^{q}(\mathbb{P}^{n}, \Omega_{\mathbb{P}^{n}}^{p}(k))\,\,\,\,\mbox{and}\,\,\,\,
h^{s}(\mathbb{P}^{n}, \bigwedge^{r}T_{\mathbb{P}^{n}}(t))=\dim H^{s}(\mathbb{P}^{n}, \bigwedge^{r}T_{\mathbb{P}^{n}}(t)),$$

\noindent we have
$$h^{q} (\mathbb{P}^{n}, \Omega_{\mathbb{P}^{n}}^{p}(k))=\left\{\begin{array}{clccc}

\binom{k+n-p}{k}\binom{k-1}{p}  & q=0,\, 0\leq p \leq n \ \ \mbox{and} \ \ k>p; \\

1 & k=0 \ \ \mbox{and} \ \ 0\leq p=q \leq n; \\

\binom{-k+p}{-k}\binom{-k-1}{n-p} & q=n,\, 0 \leq p \leq n \ \ \mbox{and} \ \ k<p-n; \\

0 & \mbox{othewise}.

\end{array}\right.$$

\noindent and

$$h^{s} (\mathbb{P}^{n}, \bigwedge^{r}T_{\mathbb{P}^{n}}(t))=\left\{\begin{array}{clccc}

\binom{t+n+r+1}{t+n+1}\binom{t+n}{n-r} & s=0,\, 0\leq r \leq n \ \ \mbox{and} \ \ t+r\geq 0; \\

1 & t=-n-1 \ \ \mbox{and} \ \ 0\leq n-r =s \leq n; \\

\binom{-t+r-1}{-t-n-1}\binom{-t-n-2}{r} & s=n,\, 0 \leq r \leq n \ \ \mbox{and} \ \ t+n+r+2 \leq 0; \\

0 & \mbox{othewise}.

\end{array}\right.$$

\bigskip

\section{Proof of Theorem \ref{thm 2}} \label{Proof}

\begin{proof} Given the foliation $\mathcal{F}$, we can consider the global section $X \in H^{0}(\mathbb{P}^{n}, T_{\mathbb{P}^{n}}(d-1))$ such that induces it. Set $1\leq r \leq n-2$. Now, we take the morphism of sheaves induced by vector field $X$, as follows
$$i_{X} : \Omega^{r}_{\mathbb{P}^{n}}(m+r+1) \longrightarrow \Omega^{r-1}_{\mathbb{P}^{n}}(m+r+d).$$

We consider the exact Koszul complex associated to the section $X$, 
\begin{equation}\label{eq 1}
0 \rightarrow \bigwedge^{n} (T_{\mathbb{P}^{n}}(d-1))^{\ast}\rightarrow \cdots \rightarrow \bigwedge^{2}(T_{\mathbb{P}^{n}}(d-1))^{\ast}\rightarrow  (T_{\mathbb{P}^{n}}(d-1))^{\ast}\rightarrow I_{Z} \rightarrow 0,
\end{equation}

\noindent where $Z$ represents the locus of singular set of the foliation $\mathcal{F},$ see \cite{Eis} and \cite{GrifHa}. 
Note that the sequence is exact since $Z$ is isolated. We can rewrite
$$0 \rightarrow \Omega^{n}_{\mathbb{P}^{n}}(-n(d-1))\rightarrow \cdots 
\rightarrow \Omega^{2}_{\mathbb{P}^{n}}(-2(d-1)) \rightarrow \Omega^{1}_{\mathbb{P}^{n}}(-(d-1))\rightarrow I_{Z} \rightarrow 0.$$

Tensorizing by $\mathcal{O}_{\mathbb{P}^{n}}(m+rd+1)$ we get
$$\begin{array}{lll}
0 \rightarrow \Omega^{n}_{\mathbb{P}^{n}}(-n(d-1)+m+rd+1)\rightarrow  \cdots \rightarrow \Omega^{r}_{\mathbb{P}^{n}}(-r(d-1)+m+rd+1)
\stackrel{i_{X}}{\longrightarrow} \cdots \\ \\ 

\cdots \rightarrow \Omega^{2}_{\mathbb{P}^{n}}(-2(d-1)+m+rd+1) \rightarrow \Omega^{1}_{\mathbb{P}^{n}}(-(d-1)+m+rd+1)\rightarrow I_{Z}(m+rd+1)\rightarrow 0. 
\end{array}$$

Now, we break this long exact sequence into some short exact sequences and use the notation $t_{l} = -l(d-1) +m+rd+1$, for $l=1,\ldots,n$.
Then we have

\begin{equation}\label{eeqq}
\begin{array}{llll}
0 \rightarrow \Omega^{n}_{\mathbb{P}^{n}}(t_{n}) \rightarrow \Omega^{n-1}_{\mathbb{P}^{n}}(t_{n-1}) \rightarrow K_{n-2} \rightarrow 0  \\
\ \ \ \ \ \ \ \ \ \ \ \ \ \ \ \ \ \ \ \ \ \vdots \vspace{1mm} \\ 
\ \ \ 0 \rightarrow K_{r} \rightarrow \Omega^{r}_{\mathbb{P}^{n}}(t_{r}) 
\stackrel{i_{X}}{\longrightarrow} K_{r-1} \rightarrow 0 \vspace{0,5mm} \\
\ \ \ \ \ \ \ \ \ \ \ \ \ \ \ \ \ \ \ \ \ \vdots \vspace{1mm} \\ 
0 \rightarrow K_{1} \rightarrow \Omega^{1}_{\mathbb{P}^{n}}(t_1) \rightarrow I_{Z}(m+rd+1)\rightarrow 0. 
\end{array}
\end{equation}

\bigskip
We need calculate the dimension $\dim H^{0}(\mathbb{P}^{n}, K_{r}),$ since each $\omega \in H^{0}(\mathbb{P}^{n}, K_{r})$ satisfies 
$i_{X}\omega=0.$ For this we take the specific long exact sequence of cohomology in $(\ref{eeqq})$
\begin{equation}\label{eqqq4}
\begin{array}{llll}
0\hspace{-1.5mm} &\rightarrow H^{0}(\mathbb{P}^{n}, K_{r+1})  \rightarrow H^{0}(\mathbb{P}^{n}, \Omega^{r+1}_{\mathbb{P}^{n}}(t_{r+1}))  \rightarrow H^{0}(\mathbb{P}^{n}, K_{r}) \rightarrow \\
&\rightarrow H^{1}(\mathbb{P}^{n}, K_{r+1}) \rightarrow H^{1}(\mathbb{P}^{n}, \Omega^{r+1}_{\mathbb{P}^{n}}(t_{r+1})) \rightarrow H^{1}(\mathbb{P}^{n}, K_{r}) \rightarrow \cdots \end{array}
\end{equation}

We need the following auxiliary result.

\begin{lemma}\label{lemma3} Under the above conditions $H^{1}(\mathbb{P}^{n}, K_{r+i}) = 0,$ for $i = 1,2,\ldots,n-1-r.$
\end{lemma}

\begin{proof} We consider two cases. First one $n=3$, implies $r=1$ and

$$H^{1}(\mathbb{P}^{3}, K_{2}) = H^{1}(\mathbb{P}^{3}, \Omega^{3}(t_{3})) = 0$$ 

\noindent by Bott's formula.

To the second one where $n>3$, take the following long exact sequences of cohomology associate in $(\ref{eeqq})$ for $i = 1,2,\ldots,n-1-r$ and $j=1,2,\ldots,n-2-r-i.$

$$\begin{array}{llllll}
\cdots \hspace{-1.5mm}&\rightarrow H^{j}(\mathbb{P}^{n}, K_{r+i+j}) \rightarrow H^{j}(\mathbb{P}^{n}, \Omega^{r+i+j}_{\mathbb{P}^{n}}(t_{r+i+j})) \rightarrow H^{j}(\mathbb{P}^{n}, K_{r+i+j-1}) \rightarrow \\
&\rightarrow H^{j+1}(\mathbb{P}^{n}, K_{r+i+j}) \rightarrow H^{j+1}(\mathbb{P}^{n}, \Omega^{r+i+j}_{\mathbb{P}^{n}}(t_{r+i+j})) \rightarrow H^{j+1}(\mathbb{P}^{n}, K_{r+i+j-1}) \rightarrow \cdots
\end{array}$$

Bott's formula implies $H^{j}(\mathbb{P}^{n}, \Omega^{r+i+j}_{\mathbb{P}^{n}}(t_{r+i+j})) = 0$, hence 
\begin{equation}\label{new1}
H^{j}(\mathbb{P}^{n}, K_{r+i+j-1}) \subset H^{j+1}(\mathbb{P}^{n}, K_{r+i+j}). 
\end{equation}

Similarly, we see that
$$\begin{array}{llllll}
\cdots \hspace{-1.5mm}&\rightarrow H^{n-r-1-i}(\mathbb{P}^{n}, \Omega^{n}_{\mathbb{P}^{n}}(t_{n}))  \rightarrow H^{n-r-1-i}(\mathbb{P}^{n}, \Omega^{n-1}_{\mathbb{P}^{n}}(t_{n-1})) \rightarrow H^{n-r-1-i}(\mathbb{P}^{n}, K_{n-2}) \rightarrow \\
&\rightarrow H^{n-r-i}(\mathbb{P}^{n}, \Omega^{n}_{\mathbb{P}^{n}}(t_{n})) \rightarrow H^{n-r-i}(\mathbb{P}^{n}, \Omega^{n-1}_{\mathbb{P}^{n}}(t_{n-1})) \rightarrow H^{n-r-i}(\mathbb{P}^{n}, K_{n-2}) \rightarrow \cdots
\end{array}$$

Bott's formula again implies $H^{n-r-1-i}(\mathbb{P}^{n}, \Omega^{n-1}_{\mathbb{P}^{n}}(t_{n-1})) = 0$, hence 

\begin{equation}\label{new2}
H^{n-r-1-i}(\mathbb{P}^{n}, K_{n-2}) \subset H^{n-r-i}(\mathbb{P}^{n},\Omega_{\mathbb{P}^{n}}^{n}(t_{n}))=0.
\end{equation}

From (\ref{new1}) and (\ref{new2}) we get

$$H^{1}(\mathbb{P}^{n}, K_{r+i}) \subset H^{2}(\mathbb{P}^{n}, K_{r+i+1})  \subset \cdots \subset H^{n-r-1-i}(\mathbb{P}^{n}, K_{n-2}) 
\subset H^{n-r-i}(\mathbb{P}^{n}, \Omega^{n}_{\mathbb{P}^{n}}(t_{n}))=0.$$

%
%
%
\end{proof}

If $m+1 \leq d$, then $h^{0}(\mathbb{P}^{n},\Omega_{\mathbb{P}^{n}}^{r+1}(t_{r+1}))=0$ by Bott's formula, 
and $h^{0}(\mathbb{P}^{n}, K_{r})=0$; this follows from Lemma \ref{lemma3} (for $i=1$) in (\ref{eqqq4}).

%
 
For $jd < m+1 \leq (j+1)d$, with $j=1,\ldots, n-r-2$, we have
$h^{0}(\mathbb{P}^{n}, \Omega_{\mathbb{P}^{n}}^{r+j+1}(t_{r+j+1}))=0$, and 
$h^{0}(\mathbb{P}^{n}, \Omega_{\mathbb{P}^{n}}^{r+s}(t_{r+s}))\neq 0$ when $s=1,\ldots, j$. 
Applying Lemma \ref{lemma3} once more to (\ref{eqqq4}), we obtain.
$$\begin{array}{rlcc}
h^{0}(\mathbb{P}^{n}, K_{r}) = & \displaystyle\sum_{i=1}^{j}(-1)^{i+1}h^{0}(\mathbb{P}^{n}, \Omega_{\mathbb{P}^{n}}^{r+i}(t_{r+i})) \\ \\

=&\displaystyle\sum_{i=1}^{j}(-1)^{i+1}\binom{m+n+1-id}{m+r+i+1-id} \binom{m+r+i-id}{r+i}.

\end{array}$$

On the other hand, when $(n-r-1)d <m+1$, we have $h^{0}(\mathbb{P}^{n}, \Omega_{\mathbb{P}^{n}}^{r+s}(t_{r+s}))\neq 0$ for $s=1,\ldots, n-r$. 
Therefore, we likewise conclude that

$$h^{0}(\mathbb{P}^{n}, K_{r}) = \sum_{i=1}^{n-r}(-1)^{i+1}\binom{m+n+1-id}{m+r+i+1-id} \binom{m+r+i-id}{r+i}.$$

\end{proof}

\section{Proof of Theorem \ref{thm 3}} \label{Proof2}

\begin{proof} We prove it as Theorem \ref{thm 2}. So we consider $\omega \in H^{0}(\mathbb{P}^{n}, \Omega^{1}_{\mathbb{P}^{n}}(m+2))$ a global section that induces the distribution $\mathcal{D}$. Take the Koszul complex associated the section $\omega$, which is exact since the distribution has only isolated singularities
$$0 \rightarrow \bigwedge^{n}(\Omega^{1}_{\mathbb{P}^{n}}(m+2))^{\ast} \rightarrow \cdots  \rightarrow \bigwedge^{2}(\Omega^{1}_{\mathbb{P}^{n}}(m+2))^{\ast} \rightarrow (\Omega^{1}_{\mathbb{P}^{n}}(m+2))^{\ast} \rightarrow I_{Z}\rightarrow 0 ,   $$

\noindent where $Z$ denotes the singular set of $\omega$. We can rewrite
\begin{equation}\label{eq A}
0 \rightarrow \bigwedge^{n}T_{\mathbb{P}^{n}}(-n(m+2)) \rightarrow \cdots  \rightarrow \bigwedge^{2}T_{\mathbb{P}^{n}}(-2(m+2)) 
\rightarrow T_{\mathbb{P}^{n}}(-(m+2)) \rightarrow I_{Z}\rightarrow 0.
\end{equation}

Tensorizing the sequence (\ref{eq A}) by $\mathcal{O}_{\mathbb{P}^{n}}(d+m+1)$ and breaks it into short exact sequences

\begin{equation}\label{eeqq2}
\begin{array}{llll}
0 \rightarrow \bigwedge^{n} T_{\mathbb{P}^{n}}(t_{n}) \rightarrow  \bigwedge^{n-1}T_{\mathbb{P}^{n}}(t_{n-1}) 
\rightarrow K_{n-2} \rightarrow 0 \vspace{0,5mm} \\
\ \ \ \ \ \ \ \ \ \ \ \ \ \ \ \ \ \ \ \ \ \vdots \vspace{1mm} \\ 
\ \ \ 0 \rightarrow K_{r} \rightarrow \bigwedge^{r}T_{\mathbb{P}^{n}}(t_{r}) \rightarrow K_{r-1} \rightarrow 0 \vspace{0,5mm} \\
\ \ \ \ \ \ \ \ \ \ \ \ \ \ \ \ \ \ \ \ \ \vdots \vspace{1mm} \\ 
0 \rightarrow K_{1} \rightarrow T_{\mathbb{P}^{n}}(d-1) \stackrel{i_{\omega}}{\longrightarrow} I_{Z}(d+m+1)\rightarrow 0,
\end{array}
\end{equation}

\bigskip

\noindent where $t_{l} = -l(m+2)+d+m+1$, for $l = 1, \ldots, n$. We are interested in the number $\dim H^{0}(\mathbb{P}^{n}, K_{1}),$ since $X \in H^{0}(\mathbb{P}^{n}, K_{1})$ implies $i_{\omega}X= \omega(X) =0$. 
For this we consider the following sequence of cohomology in (\ref{eeqq2})
\begin{equation}
\begin{array}{llll}\label{eq B}

0 \hspace{-1.5mm}&\rightarrow H^{0}(\mathbb{P}^{n}, K_{2})  \rightarrow H^{0}(\mathbb{P}^{n}, \bigwedge^{2}T_{\mathbb{P}^{n}}(t_2))  \rightarrow H^{0}(\mathbb{P}^{n}, K_{1}) \rightarrow \\

&\rightarrow H^{1}(\mathbb{P}^{n}, K_{2}) \rightarrow H^{1}(\mathbb{P}^{n}, \bigwedge^{2}T_{\mathbb{P}^{n}}(t_2)) 
\rightarrow H^{1}(\mathbb{P}^{n}, K_{1}) \rightarrow \cdots \end{array}\end{equation}

Using the Lemma 3.1 in \cite{SoaCo}, we have $H^{1}(\mathbb{P}^{n}, K_{2}) =0$. To study the vanishing of group 
$H^{0}(\mathbb{P}^{n}, K_{2})$, we associated another long exact sequences of cohomology in (\ref{eeqq2})

$$\begin{array}{llll}
0 \hspace{-1.5mm}&\rightarrow H^{0}(\mathbb{P}^{n}, K_{3})  \rightarrow H^{0}(\mathbb{P}^{n}, \bigwedge^{3}T_{\mathbb{P}^{n}}(t_{3}))  \rightarrow H^{0}(\mathbb{P}^{n}, K_{2}) \rightarrow \\ 

&\rightarrow H^{1}(\mathbb{P}^{n}, K_{3}) \rightarrow H^{1}(\mathbb{P}^{n}, \bigwedge^{3}T_{\mathbb{P}^{n}}(t_{3})) \rightarrow H^{1}(\mathbb{P}^{n}, K_{2}) \rightarrow \cdots  
\end{array}$$
%
$$\vdots$$
%
$$\begin{array}{llllll}
\cdots \hspace{-1.5mm}&\rightarrow H^{n-4}(\mathbb{P}^{n}, \bigwedge^{n}T_{\mathbb{P}^{n}}(t_{n}))  \rightarrow H^{n-4}(\mathbb{P}^{n}, \bigwedge^{n-1}T_{\mathbb{P}^{n}}(t_{n-1}))  \rightarrow H^{n-4}(\mathbb{P}^{n}, K_{n-2}) \rightarrow \\

&\rightarrow H^{n-3}(\mathbb{P}^{n}, \bigwedge^{n}T_{\mathbb{P}^{n}}(t_{n})) \rightarrow H^{n-3}(\mathbb{P}^{n}, \bigwedge^{n-1}T_{\mathbb{P}^{n}}(t_{n-1})) \rightarrow H^{n-3}(\mathbb{P}^{n}, K_{n-2}) \rightarrow \cdots 
\end{array}$$

\noindent The hypothesis $d<2(m+1)$ implies that $H^{0}(\mathbb{P}^{n}, \bigwedge^{3}T_{\mathbb{P}^{n}}(t_{3}))=0$.
Now, the hypothesis $d \neq \left(\frac{n+1}{2}\right)m+1$ implies that 
$$H^{1}(\mathbb{P}^{n}, \bigwedge^{4}T_{\mathbb{P}^{n}}(t_{4}))=H^{2}(\mathbb{P}^{n}, \bigwedge^{5}T_{\mathbb{P}^{n}}(t_{5}))=
\cdots = H^{n-4}(\mathbb{P}^{n}, \bigwedge^{n-1}T_{\mathbb{P}^{n}}(t_{n-1})) = 0.$$

\noindent In fact, the only way to have $H^{l-3}(\mathbb{P}^{n}, \bigwedge^{l}T_{\mathbb{P}^{n}}(t_{l}))\neq0$
is if $t_l=-l(m+2)+d+m+1=-n-1$ and $n-l=s=l-3$, which leads to $d = \left(\frac{n+1}{2}\right)m+1$, a contradiction.

%

Consequently, we obtain the following sequence of inclusions
$$ H^{0}(\mathbb{P}^{n}, K_{2}) \subset H^{1}(\mathbb{P}^{n}, K_{3}) \subset \cdots \subset H^{n-4}(\mathbb{P}^{n}, K_{n-2}) \subset H^{n-3}(\mathbb{P}^{n}, \bigwedge^{n}T_{\mathbb{P}^{n}}(t_{n})) = 0.$$

\noindent  Note that for $n=3$, we have 
$$H^{0}(\mathbb{P}^{3}, K_{2}) = H^{0}(\mathbb{P}^{3}, \bigwedge^{3}T_{\mathbb{P}^{3}}(t_{3})) = 0.$$

With this vanishing and (\ref{eq B}) we get $H^{0}(\mathbb{P}^{n}, K_{1}) \simeq 
H^{0}(\mathbb{P}^{n}, \bigwedge^{2}T_{\mathbb{P}^{n}}(t_{2})).$ To conclude the proof, we use the Bott formula below

$$\displaystyle \dim H^{0}(\mathbb{P}^{n}, K_{1}) = h^{0}(\mathbb{P}^{n}, 
\bigwedge^{2}T_{\mathbb{P}^{n}}(t_{2}))  = \left\{\begin{array}{cl}
\binom{d-m+n}{d-m+n-2}\binom{d-m+n-3}{n-2} &,\ \ \mbox{if}\ \ m+1 \leq d \\
0  &,\ \ \mbox{if} \ \ m+1>d.
\end{array}\right.$$
\end{proof}

\section{Examples} \label{Exa}

In this section, we show two examples to illustrate the above results.

\begin{example} Let $\mathcal{F}$ be a one-dimensional holomorphic foliation with isolated singularities on $\mathbb{P}^{3}$ 
induced by the homogeneous vector field
$$X=a_0z_1^d\frac{\partial}{\partial z_0}+a_1z_0^d\frac{\partial}{\partial z_1}+
a_2z_3^d\frac{\partial}{\partial z_2}+a_3z_2^d\frac{\partial}{\partial z_3},$$
where $a_i\neq 0$ and $d\geq 1$. Let $\mathcal{G}$ be a holomorphic distribution of codimension one on $\mathbb{P}^{3}$ induced by the $1$-form
$$\omega=FdG-GdF,$$
where $F=a_1z_0^{d+1}-a_0z_1^{d+1}$ and $G=a_3z_2^{d+1}-a_2z_3^{d+1}$. Then we have
$i_X\omega=0$ and 
$$\deg(\mathcal{F})=d< 2d=\deg(\mathcal{G}),$$
\noindent which satisfies the quote in Corollary \ref{thm 1}.
\end{example}

The next example shows that the bound given in Corollary \ref{thm 1} is sharp. 
\bigskip

\begin{example} \label{Ex2}
Let $\mathcal{F}$ be a one-dimensional holomorphic foliation with isolated singularities on $\mathbb{P}^n$ 
of degree $d\geq 1$, induced by the homogeneous vector field $X$. Let $\mathcal{G}$ be a Pfaff system, 
induced by $\omega=i_X{\omega_0} \in H^{0}(\mathbb{P}^{n}, \Omega_{\mathbb{P}^{n}}^k(d+k+1))\setminus \{0\}$, 
where $\omega_0 \in H^{0}(\mathbb{P}^{n}, \Omega_{\mathbb{P}^{n}}^{k+1}(k+2))$. 
Note that the Bott formula imply $h^{0}(\mathbb{P}^{n}, \Omega_{\mathbb{P}^{n}}^{k+1}(k+2))=\binom{n+1}{k+2}$.
Then we have $i_{X}\omega=0$ and 
$$\deg(\mathcal{F})=d=\deg(\mathcal{G}).$$
\end{example}

\section{Some inequalities for flags of foliations and distributions} \label{Ineq}

In this section, we consider flags of specific holomorphic foliations and distributions on $\mathbb{P}^{n}$ and derive inequalities related to their degrees. These inequalities are motivated by the so-called Poincar\'e problem for foliations. In particular, we deduce stability results for the tangent sheaf of some rank two holomorphic foliations. First, we are going to introduce some terminologies that will be necessary.

Let $X\subset\mathbb{P}^{n}$ be a subvariety. We denote by $\mathrm{reg}(X)$ the Castelnuovo-Mumford regularity of $X$.
Let us remember that the {\it Castelnuovo-Mumford regularity} is an invariant of arithmetic nature, a good measure of the complexity of $X$:
it is well-known that $X$ is cut out by hypersurfaces with degree at most its regularity, $\mathrm{reg}(X)$. 
However, when $X$ is arithmetically Cohen-Macaulay, for example, a complete intersection, 
the regularity acquires a more geometric meaning: cut $X$ by as many general hyperplanes as its dimension to obtain a set $S$ of points;
then the regularity of $X$ is the smallest integer $r$ such that for each $p\in S$ there is a hypersurface of degree $r-1$
passing through all the points of $S\setminus\left\{p\right\}$. For example, the regularity of a hypersurface is its degree. 
For more details about Castelnuovo-Mumford regularity, see for instance \cite{CE, Es} and \cite{Eis}. 
Finally, we say $X$ is {\it nonsingular in codimension} $1$ if $\mathrm{codim}\hspace{-0.2mm}\big(\mathrm{Sing}(X),X\big)\geq 2$. 

The space of codimension one holomorphic foliations of degree $k$ on $\mathbb{P}^{n}$ is defined by
$$\textsc{Fol}(\mathbb{P}^{n},k)=\left\{\left[\omega\right]\in\mathbb{P}H^0
\left(\mathbb{P}^{n},\Omega^1_{\mathbb{P}^{n}}(k+2)\right)\,\,|
\,\,\omega\wedge d\omega=0\,\,\,\mbox{and}\,\,\,\mathrm{codim}(\mathrm{Sing}(\omega))\geq2\right\}.$$
In \cite{Jou}, Jouanolou describes the irreducible components of $\textsc{Fol}(\mathbb{P}^{n},k)$, of degrees $k= 0, 1$:
one component when $k=0$, and two components when $k=1$, with $n\geq3$. For foliations of degree $2$ with $n\geq3$, 
Cerveau and Lins Neto \cite{CeL} showed that there are exactly six irreducible components: two logarithmic,
two rational, the exceptional, and the linear pullback. For more details on the space of foliations and its irreducible components, we refer the reader, for example, to \cite{CeL,DN,Jou,Ne,Ne1,VPC} and the references therein. 
Now, let us consider the Poincar\'e problem for flags for some classes of foliations, such as logarithmic, 
and pull-back. 

\begin{theorem}\label{A} Let $\mathcal{G}\in\mathcal{L}(n,p,d_1,\ldots,d_r)$, $r\geq p+1$, be a logarithmic foliation on 
$\mathbb{P}^{n}$, $n\geq p+2$, induced in homogeneous coordinates by a $p$-form
$$\omega=F_1\cdots F_r\cdot\hspace{-1mm}\sum_{\textsc{I}=(i_1<\cdots<i_p)}
\lambda_{\,\textsc{I}}\frac{dF_{i_1}}{F_{i_1}}\wedge\cdots\wedge\frac{dF_{i_p}}{F_{i_p}},$$
for some irreducible homogeneous polynomials $F_i$ of degree $d_i\geq 1$ and $\lambda_{\textsc{\,I}}\neq 0$.  
Set $\left|d\right|=\sum d_i$, $V=V_{i_1,\ldots,i_{p+1}}=V(F_{i_1},\ldots,F_{i_{p+1}})$ and $R=\mathrm{reg}\big(\mathrm{Sing}(V)\big)$.
Let $\mathcal{F}\prec\mathcal{G}$ be a flag of holomorphic foliations on $\mathbb{P}^{n}$. 
\vskip 0.2cm
\begin{enumerate}
	\item \label{zz} If some $\left(F_{i}=0\right)$ is smooth and $\dim(\mathcal{F})=1$, then 
$$\deg(\mathcal{G})\leq\deg(\mathcal{F})+\left|d\right|-d_i-p.$$	
	\item \label{zzz} If some $\left(F_{i}=0\right)$ is a normal crossing hypersurface and $\dim(\mathcal{F})=1$, then
$$\deg(\mathcal{G})\leq\deg(\mathcal{F})+\left|d\right|-d_i+n-p-1.$$		
\item Suppose that $V$ is a complete intersection curve $(n=p+2)$ and $\dim(\mathcal{F})=1$.
\vskip 0.2cm
	\begin{enumerate}
		\item[a)] If $V$ is a smooth and $V\not\subset\mathrm{Sing}(\mathcal{F})$, then 
$$\deg(\mathcal{G})\leq\deg(\mathcal{F})+\left|d\right|-\sum_{k=1}^{p+1} d_{i_k}.$$ 
		\item[b)] If $V$ is reduced with at most ordinary nodes as singularities, and $V\not\subset\mathrm{Sing}(\mathcal{F})$, then
$$\deg(\mathcal{G})\leq\deg(\mathcal{F})+\left|d\right|-\sum_{k=1}^{p+1} d_{i_k}+1.$$
	\end{enumerate} 
	\item Suppose that $V$ is a complete intersection and $\mathrm{codim}(\mathcal{F})=p+1$. 
\vskip 0.2cm
\begin{enumerate}
	\item[a)] Assume $V$ is reduced and 
$\mathrm{dim}\hspace{-0.2mm}\big(\mathrm{Sing}(\mathcal{F})\cap V\big)< n-p-1$. If
$R\leq\sum_{k=1}^{p+1} d_{i_k}-p-2$, then 
$$\deg(\mathcal{G})\leq\deg(\mathcal{F})+\left|d\right|-\sum_{k=1}^{p+1} d_{i_k};$$ 
and if $R>\sum_{k=1}^{p+1} d_{i_k}-p-2$, then
$$\deg(\mathcal{G})\leq \frac{1}{2}\hspace{-0.2mm}\big(\deg(\mathcal{F})+R+1\big)+\left|d\right|-\sum_{k=1}^{p+1} d_{i_k}.$$	
  \item[b)] If $V$ is nonsingular in codimension $1$ and $V\not\subset\mathrm{Sing}(\mathcal{F})$, then 
$$\deg(\mathcal{G})\leq\deg(\mathcal{F})+\left|d\right|-\sum_{k=1}^{p+1} d_{i_k}+1.$$
\end{enumerate}
\end{enumerate}
\end{theorem}
\vskip 0.2cm
Note that $V$ is a complete intersection for $p=1$.
\begin{proof}
First note that $\deg(\mathcal{G})=\left|d\right|-p-1$. Let us first prove $(1)$ and $(2)$. 
Since $\left(F_{i}=0\right)$ is invariant by $\mathcal{G}$, it follows that $\left(F_{i}=0\right)$ 
is invariant by $\mathcal{F}$. Then if $\left(F_{i}=0\right)$ is smooth we have $d_i\leq\deg(\mathcal{F})+1$; see \cite{Es,So}, 
and the assertion $(1)$ follows, and if $\left(F_{i}=0\right)$ is a normal crossing hypersurface we have $d_i\leq\deg(\mathcal{F})+n$; 
see \cite{BM}, and $(2)$ is proved. It is easy to see that $V\subset\mathrm{Sing}(\mathcal{G})$: we can write $\omega$ in the form
$$\omega=F_1\cdots F_{i_{p+1}} \cdots F_r\cdot\hspace{-1mm}
\left(\lambda_{\,\textsc{I}}\frac{dF_{i_1}}{F_{i_1}}\wedge\cdots\wedge\frac{dF_{i_p}}{F_{i_p}}
+\omega_1+\cdots+\omega_p\right),$$
where $\omega_k$ is a logarithmic $p$-form which does not contain the term $\frac{dF_{i_k}}{F_{i_k}}$.
Then it follows that $V$ is invariant by $\mathcal{F}$. Now, let us prove $(3)$. Then if $V$ is a smooth we have 
$\sum_{k=1}^{p+1} d_{i_k}\leq\deg(\mathcal{F})+p+1$; see \cite{So0}, and the assertion a) follows, and if $V$ is reduced with at most ordinary nodes as singularities we have $\sum_{k=1}^{p+1} d_{i_k}\leq\deg(\mathcal{F})+p+2$; see \cite{CC,Es}, and b) is proved. 
Finally, let us prove $(4)$. Set $\rho=R+p+2-\sum_{k=1}^{p+1} d_{i_k}$, then if $\rho\leq 0$ we have 
$\sum_{k=1}^{p+1} d_{i_k}\leq \deg(\mathcal{F})+p+1$, and if $\rho>0$ we have $\sum_{k=1}^{p+1} d_{i_k}\leq \deg(\mathcal{F})+p+1+\rho$; see 
\cite{C,CE}, and the assertion a) follows. If $V$ is nonsingular in codimension $1$, then one can take $\rho=1$; see \cite{CJ}, then  
$\sum_{k=1}^{p+1} d_{i_k}\leq \deg(\mathcal{F})+p+2$, and b) is proved. 
\end{proof}

\begin{rmk}
In item $(\ref{zz})$ above, when $n$ is odd and $\mathcal{F}$ only has isolated non-degenerates singularities, 
the bound is attained if and only if $\mathrm{Sing}(\mathcal{F})\subset (F_i=0)$. 
In fact, in \cite{CoMaa}, the authors showed that $\mathrm{Sing}(\mathcal{F})\subset (F_i=0)$ 
if and only if $d_i=\deg(\mathcal{F})+1$. Similarly, in item $(\ref{zzz})$, it follows from \cite{BM} that if the bound is attained, 
then $\mathcal{F}$ is given by a global closed logarithmic $(n-1)$-form with poles along $(F_i=0)$.
In general, since $V$ is invariant under $\mathcal{F}$, we note that the bounds in Theorem $\ref{A}$ are attained when the bounds in the Poincar\'e problem for $\mathcal{F}$ with respect to $V$ are attained.
\end{rmk}

Recall that rational foliations arise as a special case of logarithmic foliations when $r=p+1$. 
A rational foliation $\mathcal{G}\in\mathcal{R}(n,d_0,\ldots,d_k)$, with $1\leq k\leq n-2$, on $\mathbb{P}^{n}$, 
is induced in homogeneous coordinates by the $k$-form 
$$\omega=i_{\vartheta}\hspace{-0.5mm}\left(dF_0\wedge\cdots\wedge dF_k\right),$$
where $F_0,\ldots,F_k$ are irreducible homogeneous polynomials of degrees $d_0,\ldots,d_k$, respectively, and 
$\vartheta$ denotes the radial vector field.

The space of rational foliations $\mathcal{R}(n,d_0,\ldots,d_k)$, for $k>1$, was thoroughly studied in \cite{VPC}, see also \cite{Ne1}.
Similarly, the space of logarithmic foliations $\mathcal{L}(n,p,d_1,\ldots,d_r)$, for $p>1$, was extensively examined in \cite{DN}, 
see also \cite{Ne1}.

\begin{corollary} \label{Aa} With the notations of Theorem $\ref{A}$, if $\mathcal{G}$ is a logarithmic foliation such that some 
$\left(F_{i}=0\right)$ is smooth and 
$\mathrm{codim}(\mathcal{G})=n-2>\left|d\right|-d_i$ (respectively $\geq$), then $\mathcal{G}$ is stable (respectively semistable).
\end{corollary}
\begin{proof} In fact, given a global section $X\in H^0(\mathbb{P}^{n},T_{\mathbb{P}^{n}}(\deg(X)-1))$ tangent to $\mathcal{G}$, 
we can write $X=PX_0$ (in homogeneous coordinates) for some homogeneous polynomial $P$ and $X_0$ 
a holomorphic foliation (reduced), then by Theorem $\ref{A}$ $(1)$ we have
$\deg(\mathcal{G})\leq\deg(X_0)+\left|d\right|-d_i-n+2<2\deg(X_0)\leq2\left(\deg(P)+\deg(X_0)\right)=2\deg(X)$, and so 
$\mathcal{G}$ is stable.	
\end{proof}

In the context of pull-back foliations, we now present the following result. 

\begin{theorem} \label{C}
Let $\mathcal{G}=F^{\ast}(\mathcal{G}_0)\in\textsc{PB}(n,m,k,r)$ be a pull-back foliation, where 
$\mathcal{G}_0$ is a one-dimensional holomorphic foliation of degree $k$ on $\mathbb{P}^{r+1}$
and $F\hspace{-1mm}:\mathbb{P}^n\rightarrow\mathbb{P}^{r+1}$, $n\geq r+2$, a rational map of degree $m$. 
\vskip 0.2cm
\begin{enumerate}
	\item \label{xx} Let $\mathcal{F}\prec\mathcal{G}$ be a flag of holomorphic foliations on $\mathbb{P}^{n}$, with $\dim(\mathcal{F})=1$.
\begin{enumerate}
\vskip 0.2cm
	\item[a)] If $\mathcal{G}_0$ leave invariant a hypersurface $\mathcal{H}$ of degree $d$, 
such that $F^{-1}(\mathcal{H})$ is a smooth hypersurface, then 
$$\deg(\mathcal{G})\leq\left(\frac{k+r+1}{d}\right)\hspace{-1mm}
\big(\deg(\mathcal{F})+1\big)-r-1.$$
  \item[b)] If $\mathcal{G}_0$ leave invariant a hypersurface $\mathcal{H}$ of degree $d$, 
such that $F^{-1}(\mathcal{H})$ is a normal crossing hypersurface, then
$$\deg(\mathcal{G})\leq\left(\frac{k+r+1}{d}\right)\hspace{-1mm}
\big(\deg(\mathcal{F})+n\big)-r-1.$$
\end{enumerate}
\vskip 0.2cm
	\item Let $S_0\subset\mathrm{Sing}(\mathcal{G}_0)$ be an irreducible component of codimension $s\geq 2$
which is a complete intersection of hypersurfaces of degrees $d_1,\ldots,d_s$, such that $S=F^{-1}\left(S_0\right)$ 
is naturally a complete intersection. Set $\left|d\right|=\sum d_i$ and $R=\mathrm{reg}\big(\mathrm{Sing}(S)\big)$. 
Let $\mathcal{F}\prec\mathcal{G}$ be a flag of holomorphic foliations on $\mathbb{P}^{n}$, with 
$\mathrm{codim}(\mathcal{F})=s$. 
\vskip 0.2cm
\begin{enumerate}
	\item[a)] Assume $S$ is reduced and $\mathrm{dim}\hspace{-0.2mm}\big(\mathrm{Sing}(\mathcal{F})\cap S\big)< n-s$.
If $R\leq m\left|d\right|-s-1$, then 
$$\deg(\mathcal{G})\leq\left(\frac{k+r+1}{\left|d\right|}\right)\hspace{-1mm}\big(\deg(\mathcal{F})+s\big)-r-1;$$
and if $R>m\left|d\right|-s-1$, then 
$$\deg(\mathcal{G})\leq\left(\frac{k+r+1}{2\hspace{-0.5mm}\left|d\right|}\right)\hspace{-1mm}\big(\deg(\mathcal{F})+R+2s+1\big)-r-1.$$
  \item[b)] If $S$ is nonsingular in codimension $1$ and $S\not\subset\mathrm{Sing}(\mathcal{F})$, then 
$$\deg(\mathcal{G})\leq\left(\frac{k+r+1}{\left|d\right|}\right)\hspace{-1mm}\big(\deg(\mathcal{F})+s+1\big)-r-1.$$
\end{enumerate}
\end{enumerate}
\end{theorem}

\begin{proof}
First note that $\deg(\mathcal{G})=m(k+r+1)-r-1$. Let us prove $(1)$. 
Since $F^{-1}(\mathcal{H})$ is invariant by $\mathcal{G}$, it follows that $F^{-1}(\mathcal{H})$
is invariant by $\mathcal{F}$. Then if $F^{-1}(\mathcal{H})$ is smooth we have $m\cdot d\leq\deg(\mathcal{F})+1$; see \cite{Es,So}, 
and the assertion a) follows, and if $F^{-1}(\mathcal{H})$ is a normal crossing hypersurface we have $m\cdot d\leq\deg(\mathcal{F})+n$; 
see \cite{BM}, and b) is proved. Now, let us prove $(2)$. Since $S\subset\mathrm{Sing}(\mathcal{G})$, it follows that 
$S$ is invariant by $\mathcal{F}$. Set $\rho=R+s+1-m\cdot\left|d\right|$, then if $\rho\leq 0$
we have $m\cdot\left|d\right|\leq \deg(\mathcal{F})+s$, and if $\rho>0$ we have $m\cdot\left|d\right|\leq \deg(\mathcal{F})+s+\rho$; 
see \cite{C,CE}, and the assertion a) follows. If $S$ is nonsingular in codimension $1$, then one can take $\rho=1$; see \cite{CJ}, then  
$m\cdot\left|d\right|\leq \deg(\mathcal{F})+s+1$, and b) is proved. 
\end{proof}

\begin{rmk}
In item $(\ref{xx})-a)$ above, when $n$ is odd and $\mathcal{F}$ only has isolated non-degenerates singularities, 
the bound is attained if and only if $\mathrm{Sing}(\mathcal{F})\subset F^{-1}(\mathcal{H})$. 
In fact, in \cite{CoMaa}, the authors showed that $\mathrm{Sing}(\mathcal{F})\subset F^{-1}(\mathcal{H})$ 
if and only if $m\cdot d=\deg(\mathcal{F})+1$. Similarly, in item $(\ref{xx})-b)$, it follows from \cite{BM} that if the bound is attained, 
then $\mathcal{F}$ is given by a global closed logarithmic $(n-1)$-form with poles along $F^{-1}(\mathcal{H})$.
In general, since $F^{-1}(\mathcal{H})$ (or $S$) is invariant under $\mathcal{F}$, we note that the bounds in Theorem $\ref{C}$ 
are attained when the bounds in the Poincar\'e problem for $\mathcal{F}$ with respect to $F^{-1}(\mathcal{H})$ (or $S$) are attained.
\end{rmk}

The space of pull-back foliations $\textsc{PB}(n,m,k,r)$, for $r>1$ was well studied in \cite{Ne}, see also \cite{Ne1}.

\begin{corollary} Let $\mathcal{G}=F^{\ast}(\mathcal{G}_0)\in\textsc{PB}(n,m,k,1)$ be a pull-back foliation and 
$p\in\mathrm{Sing}(\mathcal{G}_0)$ be an isolated singularity. Set
$S=F^{-1}\left(\left\{p\right\}\right)$ and $R=\mathrm{reg}\big(\mathrm{Sing}(S)\big)$. Let 
$\mathcal{F}\prec\mathcal{G}$ be a flag of holomorphic foliations on $\mathbb{P}^{n}$, with 
$\mathrm{codim}(\mathcal{F})=2$. 
\vskip 0.2cm
\begin{enumerate}
	\item Assume $S$ is reduced and $\mathrm{dim}\hspace{-0.2mm}\big(\mathrm{Sing}(\mathcal{F})\cap S\big)< n-2$.
If $R\leq 2m-3$, then 
$$\deg(\mathcal{G})\leq\left(\frac{k+2}{2}\right)\hspace{-0.5mm}\deg(\mathcal{F})+k;$$
and if $R>2m-3$, then 
$$\deg(\mathcal{G})\leq\left(\frac{k+2}{4}\right)\hspace{-1mm}\big(\deg(\mathcal{F})+R+1\big)+k.$$
  \item If $S$ is nonsingular in codimension $1$ and $S\not\subset\mathrm{Sing}(\mathcal{F})$, then 
$$\deg(\mathcal{G})\leq\left(\frac{k+2}{2}\right)\hspace{-1mm}\big(\deg(\mathcal{F})+1\big)+k.$$
\end{enumerate}
\end{corollary}


\begin{corollary} Let $\mathcal{G}=F^{\ast}(\mathcal{G}_0)$ be a pull-back foliation, where 
$\mathcal{G}_0$ is a one-dimensional holomorphic foliation of degree $k$ on $\mathbb{P}^{n-1}$
and $F\hspace{-1mm}:\mathbb{P}^n\rightarrow\mathbb{P}^{n-1}$ a rational map. 
Suppose that $\mathcal{G}_0$ leave invariant a hypersurface $\mathcal{H}$ of degree $d$, 
such that $F^{-1}(\mathcal{H})$ is a smooth hypersurface. If $\mathrm{codim}(\mathcal{G})=n-2\leq 2d-k-1$, 
then $\mathcal{G}$ is semistable for $n=3$ and stable for $n>3$. 
\end{corollary}

\begin{proof} In fact, given a global section $X\in H^0(\mathbb{P}^{n},T_{\mathbb{P}^{n}}(\deg(X)-1))$ tangent to $\mathcal{G}$, 
we can write $X=PX_0$ (in homogeneous coordinates) for some homogeneous polynomial $P$ and $X_0$ 
a holomorphic foliation (reduced), then by Theorem $\ref{C}$ $(1)$ we have 
$$\deg(\mathcal{G})\leq\left(\frac{k+n-1}{d}\right)\hspace{-1mm}\big(\deg(X_0)+1\big)-(n-1).$$
Then $\deg(\mathcal{G})\leq 2\left(\deg(X_0)+1\right)-2=2\deg(X_0)\leq2\left(\deg(P)+\deg(X_0)\right)=2\deg(X)$ 
for $n=3$ and so $\mathcal{G}$ is semistable, and 	
$\deg(\mathcal{G})\leq 2\left(\deg(X_0)+1\right)-(n-1)<2\deg(X_0)\leq2\deg(X)$ for $n>3$ and so $\mathcal{G}$ is stable.
\end{proof}

The following theorem generalizes the Theorem $1.2$ in \cite{SoaCo}.

\begin{theorem} \label{D} Let $\mathcal{D}\prec\mathcal{G}$ be a flag of reduced holomorphic distributions on $\mathbb{P}^{n}$. 
Suppose that $\mathcal{G}$ is holomorphically decomposable, i.e., in homogeneous coordinates 
$\mathcal{G}$ is given by 
$$\omega=\omega_1\wedge\dots\wedge\omega_k,$$ 
for some homogeneous $1$-forms $\omega_i$. 
If the tangent sheaf of $\mathcal{D}$ is split and $\mathrm{codim}\big(\mathrm{Sing}(\mathcal{G})\big)\geq n-\mathrm{dim}(\mathcal{D})+1$, then
$$\min\hspace{-0,5mm}\left\{\deg(\omega_i)\right\}\leq\deg(\mathcal{D})+1.$$
In particular, 
\begin{enumerate}
	\item If $\mathcal{G}=\mathcal{F}_1\cap\ldots\cap\mathcal{F}_k$ is a complete intersection, i.e., 
each $\mathcal{F}_i$ is a codimension one foliation induced by $\omega_i$ (and so $\mathcal{G}$ is integrable), then
$$\min\hspace{-0,5mm}\left\{\deg(\mathcal{F}_i)\right\}\leq\deg(\mathcal{D}).$$ 
\item If all the forms $\omega_i$ have the same degree, then 
$$\deg(\mathcal{G})\leq \mathrm{codim}(\mathcal{G})\cdot\deg(\mathcal{D})+\mathrm{codim}(\mathcal{G})-1.$$
\end{enumerate}
\end{theorem} 

\begin{proof} We follow the idea in \cite{SoaCo}. Let $dV=dz_0\wedge\cdots\wedge dz_n$ and $X_1,\ldots,X_{n-d}$ be homogeneous vector fields such that $T_{\mathcal{D}}=\bigoplus T_{\mathcal{F}_{X_i}}$, where $d=\mathrm{codim}(\mathcal{D})$. 
It is well known that $\mathcal{D}$ is induced by the $d$-form 
$i_{X_1}\cdots i_{X_{n-d}}i_{\vartheta}dV$, where $\vartheta$ is the radial vector field, in particular 
$\deg(\mathcal{D})=\sum_{i=1}^{n-d} \deg(X_i)$.

Suppose $\mathcal{G}$ is given by $\omega\in H^{0}(\mathbb{P}^{n}, \Omega_{\mathbb{P}^{n}}^{k}(\deg(\mathcal{G})+k+1))$.
Since $i_{X_1}\omega=\cdots=i_{X_{n-d}}\omega=i_{\vartheta}\omega=0$, we have 
$$\left(i_{X_1}\cdots i_{X_{n-d}}i_{\vartheta}dV\right)\wedge\omega=0.$$
Now, $\left(i_{X_1}\cdots i_{X_{n-d}}i_{\vartheta}dV\right)\wedge\left(\omega_1\wedge\cdots\wedge\omega_k\right)=0$ and 
$\mathrm{codim}\big(\mathrm{Sing}(\mathcal{G})\big)\geq n-d+1$ 
allows us to use Saito's generalization of the de Rham division lemma \cite{Mal, Sa} and conclude that there exists 
homogeneous polynomials 
$(d-1)$-forms $\eta_i$ on $\mathbb{C}^{n+1}$ such that
$$i_{X_1}\cdots i_{X_{n-d}}i_{\vartheta}dV=\sum_{i=1}^k\omega_i\wedge \eta_i.$$
Since $\omega_{i_0}\wedge \eta_{i_0}\neq 0$ for some $i_0$, we have
$$\deg(\omega_{i_0})\leq\deg(\omega_{i_0})+\deg(\eta_{i_0})=\deg(\omega_{i_0}\wedge\eta_{i_0})
=\deg\left(i_{X_1}\cdots i_{X_{n-d}}i_{\vartheta}dV\right)=\deg(\mathcal{D})+1.$$
Then, $\min\hspace{-0,5mm}\left\{\deg(\omega_i)\right\}\leq\deg(\mathcal{D})+1$.
\end{proof}

Decomposable forms and complete intersection of foliations were well studied in \cite{CerLN}.

\end{document}